\documentclass[12pt, reqno]{amsart}

\usepackage{amssymb,upref,enumerate}
\usepackage{tikz}
\usepackage{mathabx}
\usepackage{amsmath,amssymb,yfonts,amscd,amsthm,amsxtra,mathrsfs}
\usepackage{latexsym}
\usepackage{soul}
\usepackage[T1]{fontenc}
\usepackage[sc,osf]{mathpazo}

\usepackage{tgtermes}
\usepackage{esint}
\usepackage{mathtools}

\usepackage[pagebackref,colorlinks=true, pdfstartview=FitV, linkcolor=red, citecolor=red, urlcolor=blue]{hyperref}

\usepackage{enumitem}

\usepackage{cite} 

\newtheorem{theorem}{Theorem}[section]

\newtheorem{corollary}[theorem]{Corollary}

\newtheorem{proposition}[theorem]{Proposition}

\theoremstyle{definition}

\newtheorem{lemma}[theorem]{Lemma}

\def\ep{{\epsilon}}

\def\al{{\alpha}}
\def\R{\mathbb R}
\def\N{\mathbb N}

\def\Z{\mathbb Z}
\def\Sn{\mathbb S}
\def\Sp{\mathcal S}

\def\ra{\rightarrow}
\def\bey{\begin{eqnarray*}}
\def\eey{\end{eqnarray*}}
 
 \allowdisplaybreaks

\def\D{{\mathscr D}}

\def\M{{\mathcal M}}

\def\XXint#1#2#3{{\setbox0=\hbox{$#1{#2#3}{\int}$}
     \vcenter{\hbox{$#2#3$}}\kern-.5\wd0}}



\thanks{The second author is supported by Simons Collaboration Grant for Mathematicians, 160427.  The third author is supported by grant  PID2023-146646NB-I00, Spanish Government; by the Basque Government through grant IT1615-22 and the BERC 2014-2017 program, and by BCAM Severo Ochoa accreditation SEV-2013-0323, Spanish Government.}

\begin{document}

\title[]{New pointwise bounds by Riesz potential type operators}
\author{Cong Hoang}

\address{Cong Hoang \\
 Department of Mathematics \\
Florida Agricultural and Mechanical University \\
 Tallahassee, FL 32307, USA}

\email{cong.hoang@famu.edu}

\author{Kabe Moen}

\address{Kabe Moen \\
 Department of Mathematics \\
 University of Alabama \\
 Tuscaloosa, AL 35487, USA}

\email{kabe.moen@ua.edu}

\author{Carlos P\'erez Moreno}

\address{Carlos P\'erez Moreno\\
 Department of Mathematics \\
 University of the Basque Country \\
 Ikerbasque and BCAM \\
 Bilbao, Spain}

\email{cperez@bcamath.org}


\maketitle

 \begin{abstract} We investigate new pointwise bounds for a class of rough integral operators, $T_{\Omega,\al}$, for a parameter $0<\al<n$ that includes classical rough singular integrals of Calder\'on and Zygmund, rough hypersingular integrals, and rough fractional integral operators.  We prove that the rough integral operators are bounded by a sparse potential operator that depends on the size of the symbol $\Omega$.  As a result of our pointwise inequalities, we obtain several new Sobolev mappings of the form $T_{\Omega,\al}:\dot W^{1,p}\ra L^q$.
 
 \end{abstract}

\section{Introduction}
Let $I_\al$ be the Riesz potential of order $0<\al<n$,
$$I_\al f(x)={c_{n,\al}}\int_{\R^n}\frac{f(y)}{|x-y|^{n-\al}}\,dy$$
where the constant $c_{n,\al}$ is chosen so that
$$\widehat{I_\al f}(\xi)=|\xi|^{-\al}\hat{f}(\xi).$$
The standard mapping properties of the Riesz potential are
\begin{equation}\label{Ialbds} \|I_\al f\|_{L^q(\R^n)}\leq C\|f\|_{L^p(\R^n)},\end{equation}
when $1<p<\frac{n}{\al}$ and $\frac1p-\frac1q=\frac{\al}{n}$.  At the endpoint $p=1$ the weak-type inequality
\begin{equation}\label{weakIalbds} \|I_\al f\|_{L^{\frac{n}{n-\al},\infty}(\R^n)}\leq C\|f\|_{L^1(\R^n)}\end{equation}
holds.
One of the main applications of the Riesz potential is the Gagliardo-Nirenberg-Sobolev inequality, 
\begin{equation}\label{GNSineq} \|f\|_{L^{p^*}(\R^n)}\leq C\|\nabla f\|_{L^p(\R^n)} \quad f\in C_c^\infty(\R^n),\end{equation}
for $1\leq p<n$ and $p^*=\frac{np}{n-p}$.  If we define the homogeneous Sobolev space $\dot W^{1,p}(\R^n)$ to be the closure of smooth functions with compact support with respect to the seminorm $\|\nabla f\|_{L^p(\R^n)}$, then we see that inequality \eqref{GNSineq} gives rise to the Sobolev embedding $$\dot W^{1,p}(\R^n)\hookrightarrow L^{p^*}(\R^n).$$
When $1<p<n$, inequality \eqref{GNSineq} follows from mapping properties of the Riesz potential $I_1$ and the pointwise inequality
\begin{equation}\label{ptwisebdd} |f(x)|\leq c_nI_1(|\nabla f|)(x), \qquad f\in C_c^\infty(\R^n). \end{equation}
Inequality \eqref{GNSineq} also holds when $p=1$ and $1^*=\frac{n}{n-1}=n'$ despite the fact that $I_1$ only satisfies the weak-type boundedness $I_1:L^1(\R^n)\ra L^{n',\infty}(\R^n)$. 

\subsection{New pointwise bounds}
Recently, in \cite{HMP}, the authors investigated an extension of inequality \eqref{ptwisebdd} to operators. More precisely, the inequality
\begin{equation}\label{Tbdd} |Tf(x)|\leq c_TI_1(|\nabla f|)(x), \quad f\in C_c^\infty(\R^n),\end{equation}
holds for several prominent operators in harmonic analysis.  We now survey these results and their consequences. 
\begin{itemize}[leftmargin=*]
\item The Hardy-Littlewood maximal operator is given by
$$Mf(x)=\sup_{Q\ni x} \fint_Q |f| $$
where the supremum is over all cubes $Q$. The operator and its iterates $M^k=M\circ\cdots\circ M$, satisfy 
\begin{equation}\label{Mkptwise} M^kf(x)\leq c_kI_1(|\nabla f|)(x), \quad k=1,2,\ldots\end{equation}
Inequality \eqref{Mkptwise} follows from the fact that $I_1(|\nabla f|)$ is an $A_1$ weight. The corresponding inequality with the smaller operator $M_1(|\nabla f|)$ on the right-hand side does not hold, even when $k=0$, despite the fact $M_1(|\nabla f|)$ is also an $A_1$ weight (see Section \ref{maxbounds}).
\item The $L^r$-maximal operators are given by
$$M_{L^r}f(x)=\sup_{Q\ni x}\left(\fint_Q |f|^r\right)^{\frac1r}=[M(|f|^r)(x)]^{\frac1r}.$$
These operators satisfy 
$$M^kf(x)\leq c_{k,r}M_{L^r}f(x),$$
for all $k=1,2,\ldots,$ and $r>1$.  If $1\leq r\leq n'$ then 
\begin{equation}\label{Lrmaxptwise} M_{L^r}f(x)\leq cI_1(|\nabla f|)(x),\end{equation}
which improves \eqref{Mkptwise}.
The value $r=n'$ is a critical index in the sense that $M_{L^{n'+\ep}}f$ is not bounded by $I_1(|\nabla f|)$ when $\ep>0$.  
\item  We can improve \eqref{Mkptwise} and \eqref{Lrmaxptwise} by considering maximal functions associated with Lorentz spaces.  Let $L^{p,q}$ be the Lorentz space normed by
\begin{equation*}\|f\|_{L^{p,q}(\R^n)} =\left(p\int_0^\infty t^q|\{x:|f(x)|>t\}|^{\frac{q}{p}}\frac{dt}{t}\right)^{\frac1q}\end{equation*}
when $1\leq q<\infty$ and
$$\|f\|_{L^{p,\infty}(\R^n)}=\sup_{t>0}t|\{x:|f(x)|>t\}|^{\frac1p}.$$
Define the normalized Lorentz average on a cube by
\begin{equation}\label{lorentzdefn}\|f\|_{L^{p,q}(Q)}=\frac{1}{|Q|^{\frac1p}}\|f\mathbf 1_Q\|_{L^{p,q}(\R^n)}\end{equation}
and the Lorentz maximal function
$$M_{L^{p,q}}f(x)=\sup_{Q\ni x}\|f\|_{L^{p,q}(Q)}.$$
The following pointwise bound holds
\begin{equation}\label{Lorentzmaxptwise} M_{L^{n',1}}f(x)\leq cI_1(|\nabla f|)(x).\end{equation}
Inequality \eqref{Lorentzmaxptwise} further improves \eqref{Lrmaxptwise} because
$$M_{L^{n'}}f(x)\leq M_{L^{n',1}}f(x)$$
by the well-known containments of Lorentz spaces $L^{n',1}\subseteq L^{n'}$.

\item Given $\Omega \in L^1(\Sn^{n-1})$ consider the rough maximal operators
$$M_\Omega f(x)=\sup_{t>0}\fint_{|y|<t}|\Omega(y')f(x-y)|\,dy$$
where $y'=y/|y|$. When $\Omega \in L^{n,\infty}(\Sn^{n-1})$ we have
$$M_\Omega f(x)\leq c\|\Omega\|_{L^{n,\infty}(\Sn^{n-1})}M_{L^{n',1}}f(x).$$
and so by inequality \eqref{Lorentzmaxptwise} we have
$$M_\Omega f(x)\leq c\|\Omega\|_{L^{n,\infty}(\Sn^{n-1})}I_1(|\nabla f|)(x).$$
\item The spherical maximal operator is defined to be
$$\mathcal Sf(x)=\sup_{r>0}\fint_{\partial B(x,r)}|f(y)|\,d\mathcal H^{n-1}(y)$$
where $\mathcal H^{n-1}$ is surface measure on $\partial B(x,r)$.  In \cite{HMP} it is shown that
\begin{equation}\label{spherical} \Sp f(x)\leq I_1(|\nabla f|)(x),\qquad f\in C_c^\infty(\R^n).\end{equation}
This inequality is of interest because it is related to the Sobolev mapping properties of $\Sp$. More general maximal operators of the form
$$\M_\mu(x)=\sup_{r>0}\int_{\R^n}|f(x+ry)|\,d\mu$$
can also be considered.  Here $\mu$ is a spherical-like measure, i.e., $\mathsf{supp}\,\mu\subseteq B(0,R)$ for some $R>0$ and
$$\mu(B(x,r))\leq Cr^{n-1} \quad x\in \R^n, r>0.$$
In fact, Haj\l{}asz and Liu \cite{HL2} show that 
$$\M_\mu f(x)\leq cI_1(|\nabla f|)(x),\quad f\in C_c^\infty(\R^n),$$
which generalizes \eqref{spherical} since $\Sp=\mathcal M_\mu$ when $\mu$ is the normalized Hausdorff measure on the unit sphere.

\item The main results from \cite{HMP} state that the pointwise bound \eqref{Tbdd} holds for rough singular integral operators 
\begin{equation}\label{roughSI} T_\Omega f(x)=\mathsf{p.v.}\int_{\R^n}\frac{\Omega(y')}{|y|^n}f(x-y)\,dy\end{equation}
where $y'=y/|y|$ and $\Omega\in L^1(\Sn^{n-1})$ satisfies $\int_{\Sn^{n-1}}\Omega=0$. It was shown that if $\Omega\in L^{n,\infty}(\Sn^{n-1})$ then 
$$|T_\Omega f(x)|\leq c\|\Omega\|_{L^{n,\infty}(\Sn^{n-1})} I_1(|\nabla f|)(x).$$
This pointwise bound was extended to more general potential operators on the right-hand side in \cite{HMP2}.

\end{itemize}

\subsection{Consequences}

By the mapping properties of $I_1$, any operator that satisfies the pointwise bound \eqref{Tbdd} will satisfy several Sobolev mapping properties.  First, the operator will be bounded from $\dot W^{1,p}(\R^n)$ to $L^{p^*}(\R^n)$ when $1<p<n$, namely,
$$\|Tf\|_{L^{p^*}(\R^n)}\leq C\|\nabla f\|_{L^p(\R^n)}.$$
When $p=1$, $I_1$ only satisfies a weak type bound and the corresponding mapping is a $\dot W^{1,1}(\R^n)\ra L^{n',\infty}(\R^n)$ inequality
\begin{equation}\label{endptLeb}\|Tf\|_{L^{n',\infty}(\R^n)}\leq C\|\nabla f\|_{L^1(\R^n)}.\end{equation}
The mapping $T:\dot W^{1,p}(\R^n)\ra L^{p^*}(\R^n)$ is less difficult to prove than $T:W^{1,p}(\R^n)\ra W^{1,p}(\R^n)$, but is still of interest, especially at the endpoint. Indeed, the endpoint inequality \eqref{endptLeb} is particularly significant for the operators $M_{L^{n',1}}$ and $\mathcal S$ which are \emph{not} bounded at the weak endpoint.  When $p=n$, we have $I_1:L^n(\R^n) \ra BMO$. The space $BMO$ does not preserve pointwise inequalities, however, $I_1$ does map into a local exponential space.  Namely, for any cube $Q$ containing the support of $f\in C_c^\infty(\R^n)$
$$\|Tf\|_{\exp L^{n'}(Q)}\leq C\|I_1(|\nabla f|)\|_{\exp L^{n'}(Q)}\leq C\left(\int_Q|\nabla f|^n\right)^{\frac1n}. $$

Muckenhoupt and Wheeden \cite{MW} showed that $I_1$ satisfies the weighted inequality
$$\|wI_1g\|_{L^{p^*}(\R^n)}\leq C\|wg\|_{L^p(\R^n)}$$
for $1<p<n$ precisely when $w\in A_{p,p^*}:$
$$[w]_{A_{p,p^*}}=\sup_Q\left(\fint_Q w^{p^*}\right)\left(\fint_Q w^{-p'}\right)^{\frac{p^*}{p'}}<\infty.$$
 When $p=1$, they showed weak type inequality 
$$\|I_1g\|_{L^{n',\infty}(w^{n'})}\leq C\|g\|_{L^1(w)}$$
holds if and only if
$$[w]_{A_{1,n'}}=\left\|\frac{M(w^{n'})}{w^{n'}}\right\|_{L^\infty(\R^n)}<\infty.$$
The sharp quantitative bounds were later found in \cite{LMPT} to be 
$$\|wI_1g\|_{L^{p^*}(\R^n)}\leq C[w]_{A_{p,p^*}}^{\frac{1}{n'}\max\{1,\frac{p'}{p^*}\}}\|wg\|_{L^p(\R^n)}$$
and
$$\|I_1g\|_{L^{n',\infty}(w^{n'})}\leq C[w]_{A_{1,n'}}^{\frac{1}{n'}}\|g\|_{L^1(w)}.$$
It follows, that any operator which satisfies $|Tf|\leq cI_1(|\nabla f|)$, will automatically satisfy Sobolev type mappings
\begin{equation*} \|wTf\|_{L^{p^*}(\R^n)}\leq C[w]_{A_{p,p^*}}^{\frac{1}{n'}\max\{1,\frac{p'}{p^*}\}}\|w\nabla f\|_{L^p(\R^n)}\end{equation*}
and
\begin{equation*} \|Tf\|_{L^{n',\infty}(w^{n'})}\leq C[w]_{A_{1,n'}}^{\frac{1}{n'}}\|\nabla f\|_{L^1(w)}.\end{equation*}

Any operator satisfying \eqref{Tbdd} will also satisfy the two weight Sobolev inequality
$$\|Tf\|_{L^q(u)}\leq C\|\nabla f\|_{L^p(v)}$$
whenever $I_1:L^p(v)\ra L^q(u)$.  The two weight inequality for $I_\al$, that is, 
$$\|I_\al f\|_{L^q(u)}\leq C\|f\|_{L^p(v)}$$ 
has a long history and we refer readers to \cite{HMP} for a discussion of the conditions that imply it. 

Lastly, it was shown in \cite{HMP} that the pointwise bound \eqref{Tbdd} implies a certain self-improvement of the form
$$M_{L^r}(Tf)(x)\leq c_rI_1(|\nabla f|)(x), \quad 0\leq r<n'.$$
In this sense, we see that there is always ``room'' in the inequality $|Tf|\leq CI_1(|\nabla f|)$ for a bigger operator on the left side.

\subsection{The main operators}

In the present work, we investigate a more general class of operators and their pointwise bounds by potential operators applied to the gradient. In particular, we extend the results from \cite{HMP} twofold. First, we consider operators with a different singularity which can be more singular than the classical case.  Second, we address the case when $\Omega$ belongs to a class below the critical index $r=n$. 

Given $0<\al<n$ and $\Omega\in L^1(\Sn^{n-1})$ with mean zero, define the rough fractional singular integral operator
\begin{equation*}\label{roughoperator} T_{\Omega,\al}f(x)=\mathsf{p.v.}\int_{\R^n} \frac{\Omega(y')}{|y|^{n+1-\al}}f(x-y)\,dy.\end{equation*}
Notice that for $\lambda>0$ if we let $f_\lambda(x)=f(\lambda x)$, then we have the following homogeneity:
$$T_{\Omega,\al}(f_\lambda)(x)=\lambda^{1-\al}T_{\Omega,\al}f(\lambda x).$$
The operators $T_{\Omega,\al}$ behave differently as $\al$ varies. 
\begin{itemize}
\item When $\al=1$, the operator $T_{\Omega,1}=T_\Omega$ is a classical rough singular integral operator as defined in \eqref{roughSI}.
\item When $0<\al<1$ the operator $T_{\Omega,\al}$ has a greater singularity than the classical Calder\'on-Zygmund operators and is known in the literature as a hypersingular integral operator, in this case, a rough version. We refer readers to the work of Wheeden \cite{Wh1,Wh2} (see also \cite{CFY}) for more on hypersingular integrals.  In this case, operator $T_{\Omega,\al}$ is related to the nonlinear fractional differential operator
$$\mathcal D^{1-\al}f(x)=\int_{\R^n}\frac{|f(y)-f(x)|}{|x-y|^{n+1-\al}}\,dy, \qquad \al\in (0,1)$$
introduced in \cite{Sp}.  In fact, if $\Omega\in L^\infty(\Sn^{n-1})$ and satisfies $\int_{\Sn^{n-1}}\Omega~=~0$ then the equality
\begin{multline*}\int_{\ep<|y|<N}\frac{\Omega(y')}{|y|^{n+1-\al}}f(x-y)\,dy\\=\int_{\ep<|y|<N}\frac{\Omega(y')}{|y|^{n+1-\al}}(f(x-y)-f(x))\,dy\end{multline*}
leads to 
$$|T_{\Omega,\al}f(x)|\leq \|\Omega\|_{L^\infty(\Sn^{n-1})}\mathcal D^{1-\al}f(x).$$
\item When $1<\al<n$ the operator $T_{\Omega,\al}$ is a rough fractional integral operator.  Such operators were used by Haj\l{}asz and Liu \cite{HL} in the study of Sobolev mappings for the spherical maximal operator $\Sp$.
\end{itemize}

As motivation, consider the case when $\Omega\in L^\infty(\Sn^{n-1})$ and $1<\al<n$.  In this case, we have 
$$|T_{\Omega,\al}f(x)|\leq c\|\Omega\|_{L^\infty(\Sn^{n-1})} I_{\al-1}(|f|)(x)$$
where the assumption $1<\al<n$ guarantees that $I_{\al-1}$ is well-defined.  For smooth functions, by inequality \eqref{ptwisebdd} we obtain the pointwise bound:
$$|T_{\Omega,\al}f(x)|\leq c I_{\al-1}(|f|)(x)\leq c I_{\al-1}\big(I_1(|\nabla f|)\big)(x)=c I_\al(|\nabla f|)(x).$$
We have used the classical convolution identity for the Riesz potential operators
$$I_\gamma\circ I_\beta=I_{\gamma+\beta}, \qquad \gamma+\beta<n,$$ which can be easily seen on Fourier transform side.  The operator $I_\al(|\nabla f|)$ has the same homogeneity as $T_{\Omega,\al}$ since
$$I_\al(|\nabla f_\lambda|)(x)=\lambda I_\al(|(\nabla f)_\lambda|)(x)=\lambda^{1-\al}I_\al(|\nabla f|)(\lambda x).$$ 
Therefore, the operator $I_\al(|\nabla f|)$ serves a natural upper bound to consider for $T_{\Omega,\al}$. Our main results will be new pointwise bounds for the operator $T_{\Omega,\al}$ depending on the size of $\Omega\in L^r(\Sn^{n-1})$ for $1<r<\infty$ and for the full range $0<\al<n$. 

For classical rough singular integral operators $T_{\Omega}$, it is well documented that the size of $\Omega$ dictates the behavior of the associated rough operator. In particular, the assumption $\Omega\in L^1(\Sn^{n-1})$ is not enough to guarantee $L^p(\R^n)$ boundedness.  The weakest assumption that is sufficient for $L^p(\R^n)$ boundedness is $\Omega\in L\log L(\Sn^{n-1})$. In our setting, the size will break into three cases:
\begin{description}
\item[Critical space] When $r=n$ and $\Omega \in L^{n,\infty}(\Sn^{n-1})$ is the critical space that appeared in \cite{HMP}.  In this case, we extend our results to the operator $T_{\Omega,\al}$ and show it is bounded by $I_\al(|\nabla \cdot|)$.  As in the case $\alpha=1$, this implies several Sobolev mappings for $T_{\Omega,\al}$, whenever $I_\al$ satisfies Lebesgue spaces bounds. 
\item[Subcritical space] When $1<r<n$, and $\Omega\in L^r(\Sn^{n-1})$ we are no longer able to obtain pointwise bounds by the Riesz potential of the gradient.  Instead, we obtain bounds by suitable sparse operators.  We derive new weighted estimates from these pointwise bounds, which we discuss in later sections.
\item[Endpoint space] At the endpoint $r=1$, we prove bounds by a sparse operator when $\Omega\in L^1(\log L)^{\frac1{n'}}(\Sn^{n-1})$, but only in the hypersingular case $0<\al<1$.  These lead to a variety of new Sobolev-type bounds.  Moreover, this is a larger class than the usual logarithmic endpoint since
$$L\log L(\Sn^{n-1})\subsetneq L(\log L)^{\frac1{n'}}(\Sn^{n-1}).$$
\end{description}
We begin with the critical space where we have the following pointwise bound.
\begin{theorem} \label{poinwisecriticalT} Suppose, $0<\al<n$, $\Omega\in L^{n,\infty}(\Sn^{n-1})$, and has mean zero.  Then 
\begin{equation}\label{pointwisecritT}|T_{\Omega,\al}f(x)|\leq c\|\Omega\|_{L^{n,\infty}(\Sn^{n-1})}I_\al(|\nabla f|)(x),\qquad f\in C_c^\infty(\R^n).\end{equation}
\end{theorem}

Inequality \eqref{pointwisecritT} leads to several Sobolev mapping properties for $T_{\Omega,\alpha}$. In particular, one obtains  
$$
\|T_{\Omega,\alpha}f\|_{L^q(u)} \leq C \|\nabla f\|_{L^p(v)},
$$
whenever the Riesz potential satisfies $I_\alpha : L^p(v) \to L^q(u)$. Remarkably, a weak-type endpoint Sobolev inequality for measures completely characterizes when general operators admit a pointwise bound by $I_\alpha(|\nabla \cdot|)$.

\begin{theorem}\label{characterization} Suppose $0<\al<n$ and $T$ is an operator defined pointwise on $C_c^\infty(\R^n)$.  Then the following are equivalent:
\begin{enumerate}
\item The pointwise inequality 
$$|Tf(x)|\leq cI_\al(|\nabla f|)(x),$$
holds for all $x\in \R^n$ and $f\in C_c^\infty(\R^n)$.
\item The weak endpoint Sobolev inequality 
\begin{equation}\label{endpoint}\|Tf\|_{L^{\frac{n}{n-\al},\infty}(\mu)}\leq C\int_{\R^n}|\nabla f|(M\mu)^{1-\frac\al{n}}\end{equation}
holds for every locally finite Borel measure $\mu$ and $f\in C_c^\infty(\R^n)$.
\end{enumerate}
\end{theorem}

We note that inequality \eqref{endpoint} arises naturally in this context. Indeed, by letting $d\mu = w\,dx$ with $w \in A_1$ (i.e., $Mw \leq c w$), inequality \eqref{endpoint} yields  
\begin{equation}\label{weakendpt}
\|Tf\|_{L^{\frac{n}{n-\alpha},\infty}(w)} \leq C \int_{\mathbb{R}^n} |\nabla f|\, w^{1 - \frac{\alpha}{n}}.
\end{equation}
Through off-diagonal extrapolation and appropriate renormalization of the weights, inequality \eqref{weakendpt} in turn implies the weighted Sobolev inequality  
$$
\|w Tf\|_{L^q(\mathbb{R}^n)} \leq C \|w \nabla f\|_{L^p(\mathbb{R}^n)},
$$
for $1 < p < \frac{n}{\alpha}$ and $\frac{1}{q} = \frac{1}{p} - \frac{\alpha}{n}$, provided that $w \in A_{p,q}$ (see Section~\ref{weights}). Further details may be found in \cite{LMPT}.

To state our results for $1<r<n$ we need to introduce some common machinery concerning dyadic cubes and sparse families.  Recall a dyadic grid $\mathscr D$ is a collection of cubes in $\R^n$ such that every cube has sidelength $2^k$ for some $k\in \Z$, for each fixed $k\in \Z$ the cubes of length $2^k$ partition $\R^n$, and the entire collection $\mathscr D$ satisfies the ``nested or disjoint'' property that $Q\cap P\in \{\varnothing,P,Q\}$ for all $Q,P\in \mathscr D$.  We say a subfamily of dyadic cubes $\mathscr S\subseteq \D$ is sparse if for each $Q\in \mathscr S$ there exists $E_Q\subseteq Q$ such that $|Q|\leq 2|E_Q|$ and the family $\{E_Q:Q\in\mathscr S\}$ is pairwise disjoint.

Given, $0<\al<n$, a sparse family $\mathscr S$, we define the sparse fractional integral operator
$$I_{\al}^\mathscr S f(x)=\sum_{Q\in \mathscr S}\ell(Q)^\al \left(\,\fint_Q f\right)\mathbf 1_Q(x).$$
In \cite{CM2} it is shown that $I_\al$ is bounded by finitely many sparse operators.  If $f\in L_c^\infty(\R^n)$ and $f\geq 0$, then there exists sparse families of cubes $\mathscr S_1,\ldots,\mathscr S_N$ such that
\begin{equation}\label{sparsebdd} I_\al f(x)\leq c\sum_{k=1}^N I_\al^{\mathscr S_k}f(x).\end{equation}
The sparse families depend on $f$, but the implicit constants in the inequality do not depend on $f$. The sparse operators are often simpler to work with and several norm inequalities can be gleaned from them.  

If we combine Theorem \ref{poinwisecriticalT} with inequality \eqref{sparsebdd} we obtain the following corollary, which provides a preview of our forthcoming results.

\begin{corollary} Suppose, $0<\al<n$, $\Omega\in L^{n,\infty}(\Sn^{n-1})$, and has mean zero.  Then given $f\in  C_c^\infty(\R^n)$, there exists finitely many sparse families of cubes $\mathscr S_1,\ldots,\mathscr S_N$ such that
$$|T_{\Omega,\al}f(x)|\leq c\|\Omega\|_{L^{n,\infty}(\Sn^{n-1})}\sum_{k=1}^NI^{\mathscr S_k}_\al(|\nabla f|)(x).$$
\end{corollary}

  In the subcritical setting for $T_{\Omega,\al}$ we do not obtain bounds by the sparse operator $I^\mathscr S_\al(|\nabla f|)$, instead, we obtain a bigger operator on the right-hand side. Given $0<\al<n$ and an exponent $1\leq s<\frac{n}{\al}$ we define the $L^s$ sparse fractional operator
$$
I_{\al,L^s}^\mathscr S f(x)=\sum_{Q\in \mathscr S}\ell(Q)^\al \left(\,\fint_Q |f|^s\right)^{\frac1s}\mathbf 1_Q(x).
$$
When $s=1$ we simply write $I_{\al,L^1}^\mathscr S=I^\mathscr S_\al.$ The restriction $s<\frac{n}{\al}$ is necessary, because the operator $I_{\al,L^s}^\mathscr S$ may not be well-defined when $s\geq \frac{n}{\al}$. Indeed, consider the sparse family $\mathscr S=\{[0,2^k)^n:k\in \N\}.$  If $s\geq \frac{n}{\al}$ then for $x\in [0,1]^n$ we have
$$I^\mathscr S_{\al,L^s}(\mathbf 1_{[0,1]^n})(x)=\sum_{Q\in \mathscr S}\ell(Q)^\al \Big(\frac{1}{|Q|}\Big)^{\frac1s}=\sum_{k=1}^\infty 2^{k(\al-\frac{n}{s})}=\infty.$$

The operators $I_{\al,L^s}^\mathscr S$ govern the rough operators $T_{\Omega,\al}$ in the subcritical case.  Our scope extends beyond the class $L^{r}(\Sn^{n-1})$ -- specifically, we are able to consider the more refined class where $\Omega$ belongs to the Lorentz space $L^{r,r^*}(\Sn^{n-1})$.  This space is normed by
$$\|\Omega\|_{L^{r,r^*}(\Sn^{n-1})}=\left(r\int_0^\infty t^{r^*}\sigma\big(\{y'\in \Sn^{n-1}:|\Omega(y')|>t\})^{\frac{r^*}{r}}\frac{dt}{t}\right)^{\frac1{r^*}},$$
where $r^*=\frac{nr}{n-r}$.

\begin{theorem} \label{rlessnptwisebd} Suppose $1<r<n$, $0<\al<1+\frac{n}{r'}$, and $\Omega \in L^{r,r^*}(\Sn^{n-1})$ has mean zero. Then there exists finitely many sparse families $\mathscr S_k \subseteq \D_k$, $k=1,\ldots, N$ such that
$$|T_{\Omega,\al}f(x)|\leq c\|\Omega\|_{L^{r,r^*}(\Sn^{n-1})}\sum_{k=1}^N I_{\al,L^s}^{\mathscr S_k}(|\nabla f|)(x)$$
where $\frac1s=\frac{1}{n}+\frac1{r'}$.
\end{theorem}

Since $r^*>r$ we have 
$$L^r(\Sn^{n-1})=L^{r,r}(\Sn^{n-1})\subseteq L^{r,r^*}(\Sn^{n-1})$$
and hence the statement of the theorems holds under the less general assumption $\Omega\in L^r(\Sn^{n-1})$.  If $r$ is a given exponent satisfying $1<r<n$, then the condition $s=\frac{r'n}{n+r'}<\frac{n}{\al}$ places a restriction on $\al$, namely $\al<1+\frac{n}{r'}$.  

We will derive the weighted estimates for the operators $I_{\al,L^s}^\mathscr S$ in Section \ref{weights}.  These weighted estimates will imply new Sobolev type bounds for $T_{\Omega,\al}$.  Before moving on, let us make a few quick observations.  The operator, $I_{\al,L^s}^\mathscr S$ is related to operator
$$f\mapsto \big[I_{\al s}(|f|^s)\big]^{\frac1s}.$$
The sparse bounds from inequality \eqref{sparsebdd} imply the pointwise relationship 
\begin{equation}\label{integralsparsebd} \big[I_{\al s}(|f|^s)\big]^{\frac1s}\simeq \left[\sum_{k=1}^N I_{\al s}^{\mathscr S_k}(f^s)\right]^{\frac1s}\leq \sum_{k=1}^N I^{\mathscr S_k}_{\al,L^s}f\end{equation}
where we have used convexity to pull the power $1/s$ into the sums.  We also have the following norm bounds, which will be used later to prove weighted estimates.
\begin{theorem} \label{Ainftybdds} Suppose $1\leq s<\frac{n}{\al}$, $0<p<\infty$, and $w\in A_\infty$.  Then there exists $C>0$ such that for any sparse family of cubes $\mathscr S$
$$\|I_{\al,L^s}^\mathscr S f\|_{L^p(w)}\leq C\|[I_{\al s}(f^s)]^{\frac1s}\|_{L^p(w)},$$
and
$$\|I_{\al,L^s}^\mathscr S f\|_{L^{p,\infty}(w)}\leq C\big\|[I_{\al s}(f^s)]^{\frac1s}\|_{L^{p,\infty}(w)},$$
for measurable functions $f\geq 0$. 
\end{theorem}

By a simple rescaling  argument we see that the operator $f\mapsto [I_{\al s}(f^s)]^{\frac1s}$ satisfies $L^p(\R^n) \ra L^q(\R^n)$ bounds when  $s<p<\frac{n}{\al}$ and $\frac1p-\frac1q=\frac{\al}{n}$ and the weak type bounds $L^s(\R^n) \ra L^{\frac{sn}{n-\al s},\infty}(\R^n)$ at the endpoint $p=s$. 
When $1<r<n$ and $\Omega\in {L^{r,r^*}}(\Sn^{n-1})$ has mean zero, by Theorem \ref{rlessnptwisebd} we have
\begin{equation*}|T_{\Omega,\al} f(x)|\leq c\sum_{k=1}^N I^{\mathscr S_k}_{\al,L^\frac{r'n}{r'+n}}(|\nabla f|)(x).\end{equation*}
The resulting weighted Sobolev estimates for $T_{\Omega,\al}$ are more complicated in general because the class of weights depends on $r'$ and $n$.  Instead, we state the following power weight result.

\begin{theorem} \label{roughpower1} Suppose $1<r<n$ and $\Omega\in L^{r,r^*}(\Sn^{n-1})$ has mean value zero and $0<\al<1+\frac{n}{r'}$.  Then the following inequality holds 
$$\left(\int_{\R^n}\big(|x|^\lambda|T_{\Omega,\al}f(x)|\big)^q\,dx\right)^{\frac1q}\leq C\left(\int_{\R^n}\big(|x|^\lambda|\nabla f(x)|\big)^p\,dx\right)^{\frac1p}$$
for $\frac{r'n}{r'+n}<p<\frac{n}{\al}$, $\frac1q=\frac1p-\frac{\al}{n}$ and
$$\al-\frac{n}{p}<\lambda<1+\frac{n}{r'}-\frac{n}{p}.$$
\end{theorem}

The operators $T_{\Omega,\al}$ become more difficult to analyze when $\Omega \in L^1(\Sn^{n-1})$.  For an example of this phenomenon see the example of Honz\'ik \cite{Hon1} for $T_\Omega=T_{\Omega,1}$. In fact most results assume $\Omega \in L^1\log L(\Sn^{n-1})$ (see \cite{Graf}).  In our context we notice that when $r=1$ the restriction $\al<1+\frac{n}{r'}$ means we can only consider the hypersingular case $0<\al<1$.  In this case, we can work with the class $L^1(\log L)^{\frac1{n'}}(\Sn^{n-1})$. 

\begin{theorem} \label{LlogLendpt} Suppose $0<\al<1$ and $\Omega \in L^1(\log L)^{\frac1{n'}}(\Sn^{n-1})$ has mean zero.  Then there exist sparse families $\mathscr S_k\subseteq \D_k$, $k=1,\ldots,N$ such that
\begin{equation*} |T_{\Omega,\al}f(x)|\leq c\|\Omega\|_{ L^1(\log L)^{\frac1{n'}}(\Sn^{n-1})} \sum_{k=1}^N I^{\mathscr S_k}_{\al, L^n}(|\nabla f|)(x).\end{equation*}
\end{theorem}

The operators $I^\mathscr S_{\al,L^n}$ satisfy $L^p(\R^n) \ra L^q(\R^n)$ mappings for $n<p<\frac{n}{\al}$ and $\frac1q=\frac1p-\frac{\al}{n}$, where note that $p=n$ corresponds to the weak endpoint.  The following mapping properties for $T_{\Omega,\al}$ now follow, which seem completely new.

\begin{theorem} Suppose $0<\al<1$ and $\Omega\in L^1(\log L)^{\frac{1}{n'}}(\Sn^{n-1})$ has mean zero. Then 
$$T_{\Omega,\al}:\dot W^{1,p}(\R^n)\ra L^{q}(\R^n) $$
for $n<p<\frac{n}{\al}$ and $\frac1q=\frac1p-\frac{\al}{n}$ and 
$$T_{\Omega,\al}:\dot W^{1,n}(\R^n)\ra L^{\frac{n}{1-\al},\infty}(\R^n)$$
\end{theorem}

Finally, we may use Theorem \ref{LlogLendpt} to obtain the weighted norm bounds from Section \ref{weights} to derive power weighted norm estimates for the hypersingular rough integral $T_{\Omega,\al}$ for $\Omega \in L(\log L)^{\frac1{n'}}(\Sn^{n-1})$.
\begin{theorem} \label{roughpower2} Suppose $0<\al<1$ and $\Omega\in L(\log L)^{\frac1{n'}}(\Sn^{n-1})$ with mean value zero.  Then the following inequality holds  
$$\left(\int_{\R^n}\big(|x|^\lambda|T_{\Omega,\al}f(x)|\big)^q\,dx\right)^{\frac1q}\leq C\left(\int_{\R^n}\big(|x|^\lambda|\nabla f(x)|\big)^p\,dx\right)^{\frac1p}$$
for $n<p<\frac{n}{\al}$, $\frac1q=\frac1p-\frac{\al}{n}$ and
$$\al-\frac{n}{p}<\lambda<1-\frac{n}{p}.$$
\end{theorem}

\subsection{Plan of paper} 

In the next section, Section \ref{prelim}, we will collect some of the necessary background needed to prove our results. In Section \ref{pointwisebds}, we will prove all our pointwise bounds from Theorems \ref{poinwisecriticalT}, \ref{rlessnptwisebd}, and \ref{LlogLendpt} along with Theorem \ref{characterization}. In Section \ref{maxbounds} we study bounds for maximal functions which include $M_\Omega$ and other maximal functions with cancellation. Section \ref{weights} will be devoted to the study of the mapping properties of the dominating operators.  While bounds for $I_\al$ and $M_\al$ are well-known, it seems that the bounds for the $L^s$ versions of these operators are not known, particularly, weighted inequalities.  

\subsubsection*{Acknowledgement} We would like to thank the anonymous referee for several valuable comments that improved the manuscript.

\section{Preliminaries}\label{prelim}
A key tool in our proofs of the pointwise bounds is the local Poincar\'e-Sobolev inequality
\begin{equation}\label{poincare}\left(\fint_Q|f-f_Q|^{q}\right)^{\frac1{q}}\leq C\ell(Q)\left(\fint_Q |\nabla f|^p\right)^{\frac1p}\end{equation}
which holds for $1\leq p<n$, $1\leq q\leq p^*$, and $f\in C^1(Q)$.  We will use an improved version in the scale of Lorentz spaces
\begin{equation}\label{Lorpoincare}\|f-f_Q\|_{L^{p^*,p}(Q)}\leq C\ell(Q)\left(\fint_Q |\nabla f|^p\right)^{\frac1p}\end{equation}
for the same values $1\leq p<n$.  Here the $\|\cdot\|_{L^{p^*,p}(Q)}$ is the Lorentz average over the cube $Q$ as defined in \eqref{lorentzdefn}.  The improvement is due to O'Neill \cite{On} and Peetre \cite{Pe}, but a simpler proof can be found in \cite{MP}. 

When $p=n$ inequality \eqref{poincare} does not hold, but the so-called Trudinger inequality holds, namely,
\begin{equation}\label{trud}\|f-f_Q\|_{\exp L^{n'}(Q)}\leq C\left(\int_Q|\nabla f|^n\right)^{\frac1n}\end{equation}
where again $\|\cdot \|_{\exp L^{n'}(Q)}$ is the the $\exp L^{n'}$ average on $Q$.  We also remark that any of the inequalities \eqref{poincare}, \eqref{Lorpoincare}, or \eqref{trud} hold for balls instead of cubes with the proper adjustment of using the radius instead of the sidelength.  We refer to \cite{MacP} for a proof of \eqref{trud} in a general framework that does not use smoothness. 

We will also use the dyadic machinery, particularly, Lemma \ref{thirdtrick}. Given a dyadic grid $\D$ and $k\in \Z$ let
$$\D^k=\{Q\in \D:\ell(Q)=2^k\}.$$
We will use the well-known collection of dyadic grids for $t\in \{0,\frac13\}^n$ defined by
\begin{equation}\label{dyadicgrids} \D_t=\{2^k([0,1)^n+m+(-1)^kt):k\in \Z,m\in \Z^n\}.\end{equation}
These dyadic grids are important because every cube can be ``approximated" by a dyadic cube from one of these grids.  Namely, we have the following lemma.
\begin{lemma}\label{thirdtrick} Given any cube $Q$, there exists $t\in \{0,\frac13\}^n$ and $Q_t\in \D_t$ such that $Q\subseteq Q_t$ and $\ell(Q_t)\leq 6\ell(Q).$
\end{lemma}
Recall that the cubes in $\D^k$ for a fixed $k$ form a partition of $\R^n$.  We have the following variant of Lemma \ref{thirdtrick} which will be needed later.  

\begin{lemma} \label{dyadicthird} Suppose $Q$ is a cube with $\ell(Q)=2^k$ for some $k$.  Then there exists $t\in \{0,\frac13\}^n$ and $Q_t\in \D_t^{k+3}$ such that $Q\subseteq Q_t$.
\end{lemma}

\begin{proof} Given the cube $Q$ with $\ell(Q)=2^k$, by Lemma \ref{thirdtrick} there exists $P_t\in \D_t$ such that $Q\subseteq P_t$ and $\ell(P_t)\leq 6\ell(Q)\leq 2^{k+3}$.  Let $Q_t$ be the unique cube in $\D_t^{k+3}$ containing $P_t$ and hence also $Q$.
\end{proof}

We will also make use of the dyadic fractional integral operator associated with a dyadic grid $\D$
$$I_\al^\D f=\sum_{Q\in \D} \ell(Q)^\al \left(\fint_Q f\right)\mathbf 1_Q.$$
Notice the difference between $I_\al^\D$ and $I_\al^\mathscr S$ is that the sum is over all dyadic cubes from $\D$, not just a sparse family $\mathscr S$.  In \cite{CM2} it is shown that the discrete operator $I_\al^\D$ and the integral operator $I_\al$ are pointwise equivalent, this is in contrast to the singular integral case.  We have the following proposition from \cite{CM2}.

\begin{proposition}\label{dyadicfracbdds} Suppose $0<\al<n$. If $\D$ is any dyadic grid and $f\geq 0$, then we have
$$I_\al^\D f(x)\leq c I_\al f(x),$$
and if $\big\{\D_t:t\in \{0,\frac13\}^n\big\}$ are the dyadic grids defined in \eqref{dyadicgrids}, then
$$I_\al f(x)\leq c \sum_{t\in \{0,\frac13\}^n} I_\al^{\D_t}f(x).$$
If $f\in L^\infty_c(\R^n)$ and $f\geq 0$, then given any dyadic grid $\D$, there exists a sparse subset $\mathscr S\subseteq \D$ such that
$$I_\al^\D f(x)\leq c I_\al^\mathscr S f(x).$$
\end{proposition}
We will need the following lemma which is similar to Proposition \ref{dyadicfracbdds} but for the $L^s$ averages. Given a dyadic grid $\D$ define the $L^s$ dyadic fractional operator  for $f\geq 0$ by
$$I_{\al,L^s}^\D f=\sum_{Q\in \D}\ell(Q)^\al\Big(\fint_Q f^s\Big)^{\frac1s}\mathbf 1_Q.$$
The proof is similar to the sparse bound proof in Proposition \ref{dyadicfracbdds}, but we include it here for completeness.

\begin{lemma}\label{Lsdyadicfracbdds} Suppose $1\leq s<\frac{n}{\al}$ and $f\in L^\infty_c(\R^n)$ satisfies $f\geq 0$. Given any dyadic grid $\D$, there exists a sparse family $\mathscr S=\mathscr S(f)\subseteq \D$ such that
$$I_{\al,L^s}^\D f(x)\leq c I_{\al,L^s}^\mathscr S f(x).$$
\end{lemma}
\begin{proof} Let $a>1$ be a number to be chosen later.  For each $k\in \Z$ let 
$$\mathscr C^k=\left\{Q\in \D: a^k<\Big(\fint_Q f^s\Big)^{\frac1s}\leq a^{k+1}\right\}.$$
Then every cube $Q\in \D$ that contributes to the sum of $I_{\al,L^s}^\D f$ belongs to $\mathscr C^k$ for some $k$ and hence
$$ I_{\al,L^s}^\D f=\sum_{Q\in \D}\ell(Q)^\al\Big(\fint_Q f^s\Big)^{\frac1s}\mathbf 1_Q\leq a\sum_{k\in \Z} a^k\sum_{Q\in \mathscr C^k} \ell(Q)^\al \mathbf 1_Q.$$
Now let $\mathscr S^k$ be the collection of all maximal dyadic cubes such that 
$$\Big(\fint_Q f^s\Big)^{\frac1s}>a^k.$$
If $\mathscr C^k$ is not empty, then such maximal cubes exist because $f\in L^\infty_c(\R^n)$. The collection of cubes in $\mathscr S^k$ is pairwise disjoint (for a fixed $k$) and satisfies
$$\Big(\fint_Q f^s\Big)^\frac1s\leq 2^{\frac{n}{s}}a^k$$
by maximality. Moreover, every cube $P\in \mathscr C^k$ is a subset of a unique $Q\in \mathscr S^k$.  Hence
 \begin{align*}
 a\sum_{k\in \Z} a^k\sum_{Q\in \mathscr C^k} \ell(Q)^\al \mathbf 1_Q&\leq a\sum_{k\in \Z} a^k\sum_{Q\in \mathscr S^k} \sum_{{P\in \mathscr C^k \atop P\subseteq Q}}\ell(P)^\al \mathbf 1_P \\ 
 &\leq a\sum_{k\in \Z} a^k\sum_{Q\in \mathscr S^k} \sum_{j=0}^\infty\sum_{{P\in \D(Q)\atop \ell(P)=2^{-j}\ell(Q)}} \ell(P)^\al \mathbf 1_P\\
 &=\frac{a}{1-2^{-\al}}\sum_{k\in \Z} a^k\sum_{Q\in \mathscr S^k} \ell(Q)^\al\mathbf 1_Q\\
 &\leq c_\al \sum_{k\in \Z}\sum_{Q\in \mathscr S^k}\ell(Q)^\al\Big(\fint_Q f^s\Big)^{\frac1s}\mathbf 1_Q.
  \end{align*}
If we let $\mathscr S=\bigcup_k \mathscr S^k$ then the desired result will follow provided we can prove that $\mathscr S$ is a sparse family.  To see this fix $Q\in \mathscr S^k$, then the sets are nested in the following way: every $P\in \mathscr S^{k+1}$ satisfies $P\subseteq Q$ for some $Q\in \mathscr S^k$. Hence if we let $\Lambda^k=\bigcup\{Q:Q\in \mathscr S^k\}$ then we have
\begin{multline*}|Q\cap \Lambda^{k+1}|=\sum_{{P\in \mathscr S^{k+1} \atop P\subseteq Q}}|P|<\frac{1}{a^{(k+1)s}}\sum_{{P\in \mathscr S^{k+1} \atop P\subseteq Q}}\int_P f^s 
\leq \frac{1}{a^{(k+1)s}}\int_Q f^s\leq \frac{2^n}{a^s}|Q|. \end{multline*}
By setting $a=2^{\frac{n+1}{s}}$ we arrive at $|Q\cap \Lambda^{k+1}|\leq \frac12|Q|$ and
the sparse condition now follows by setting $E_Q=Q\backslash \Lambda^{k+1}$.
\end{proof}

We will also need some basic facts about Orlicz spaces.  Given a Young function $\Phi(t)$ we define the Orlicz average of a measurable set $E$ (usually a cube or a ball),
$$\|u\|_{L^\Phi(E)}=\inf\left\{\lambda>0:\fint_E\Phi\Big(\frac{|u(x)|}{\lambda}\Big)\,dx\leq 1\right\}.$$
The associate space is defined by the Young function 
$$\widebar{\Phi}(t)=\sup_{s>0}(st-\Phi(s)),$$
which satisfies $\Phi^{-1}(t)\widebar{\Phi}^{-1}(t)\simeq t$.  We also have the following H\"older inequality for Young functions
\begin{equation}\label{orliczholder} \fint_E |fg|\,\leq c\|f\|_{L^\Phi(E)}\|g\|_{L^{\widebar{\Phi}}(E)}.\end{equation} We will be particularly interested in exponential Young functions of the form $\Phi(t)=\exp(t^q)-1$ for some $q>1$.  In this case we have
$$\widebar{\Phi}(t)\simeq t\log(1+t)^{\frac1{q}}$$
and inequality \eqref{orliczholder} becomes
\begin{equation}\label{logexpholder} \fint_E |fg|\,\leq c\|f\|_{L(\log L)^{\frac1{q}}(E)}\|g\|_{\exp L^q(E)}.\end{equation} 
We will also need the following H\"older's inequality in Lorentz spaces,
\begin{equation}\label{lorentzholder}\fint_E|fg|\leq c\|f\|_{L^{p,q}(E)}\|g\|_{L^{p',q'}(E)}\end{equation}
which holds for all $1<p<\infty$ and $1\leq q\leq \infty$.

\section{Pointwise bounds}\label{pointwisebds} In this section we prove our pointwise bounds for the integral operators $T_{\Omega,\al}$. To prove this, as mentioned above, we will use the local Poincar\'e-Sobolev inequality.  The proofs of the bounds in Theorems \ref{poinwisecriticalT}, \ref{rlessnptwisebd}, and \ref{LlogLendpt} all begin with the same decomposition.  

We split $T_{\Omega,\al}$ into dyadic annuli
\begin{multline*}T_{\Omega,\al}f(x)=\sum_{k\in \Z}\int_{2^{k-1}<|y|\leq 2^k}\frac{\Omega(y')}{|y|^{n+1-\al}} f(x-y)\,dy \\ =\sum_{k\in \Z}\int_{2^{k-1}<|y|\leq 2^k}\frac{\Omega(y')}{|y|^{n+1-\al}} (f(x-y)-f_{B(x,2^k)})\,dy\end{multline*}
where we have used the cancellation $\int_{\Sn^{n-1}}\Omega=0$ to subtract the constant $$f_{B(x,2^k)}=\fint_{B(x,2^k)} f.$$ The decomposition leads to 
\begin{multline*}|T_{\Omega,\al}f(x)|\leq \sum_{k\in \Z}\int_{2^{k-1}<|y|\leq 2^k}\frac{|\Omega(y')|}{|y|^{n+1-\al}} |f(x-y)-f_{B(x,2^k)}|\,dy \\ \leq c_{n,\al}\sum_{k\in \Z}2^{k(\al-1)}\fint_{|y|\leq 2^k}{|\Omega(y')|} |f(x-y)-f_{B(x,2^k)}|\,dy.\end{multline*}
We now analyze the term 
$$\fint_{|y|\leq 2^k}{|\Omega(y')|} |f(x-y)-f_{B(x,2^k)}|\,dy.$$
We now divide this into to the three cases 
\begin{description}
\item[Case 1] the critical case: $\Omega \in L^{n,\infty}(\Sn^{n-1})$.
\item[Case 2] the subcritical case: $\Omega \in {L^{r,r^*}}(\Sn^{n-1})$ for $1<r<n$.
\item[Case 3] the endpoint case: $\Omega \in L(\log L)^{\frac1{n'}}(\Sn^{n-1})$.
\end{description}

{\bfseries Case 1:} We first use the Lorentz H\"older's inequality \eqref{lorentzholder} and then the local Poincar\'e-Sobolev inequality \eqref{Lorpoincare} with $p=1$ and $1^*=n'$ to get
\begin{align*}\lefteqn{\fint_{|y|\leq 2^k}{|\Omega(y')|} |f(x-y)-f_{B(x,2^k)}|\,dy} \\ &\qquad\qquad \leq \|\Omega(\cdot/|\cdot|)\|_{L^{n,\infty}(B(0,2^k))}\|f(x-\cdot)-f_{B(x,2^k)}\|_{L^{n',1}(B(0,2^k))}\\ &\qquad\qquad\simeq c_n\|\Omega\|_{L^{n,\infty}(\Sn^{n-1})}\|f-f_{B(x,2^k)}\|_{L^{n',1}(B(x,2^k))} \\
&\qquad\qquad\leq c_n\|\Omega\|_{L^{n,\infty}(\Sn^{n-1})}2^k\fint_{B(x,2^k)}|\nabla f|. \end{align*}
In the above estimates we have used that 
\begin{equation}\label{levelset} |\{y\in B(0,2^k):|\Omega(y/|y|)|>t\}|=\frac{2^{kn}}{n}\,\sigma(\{y'\in \Sn^{n-1}:|\Omega(y')|>t\})\end{equation}
where $\sigma=\mathcal H^{n-1}$ to conclude that the averages satisfy
$$\|\Omega(\cdot/|\cdot|)\|_{L^{n,\infty}(B(0,2^k))}=c\|\Omega\|_{L^{n,\infty}(\Sn^{n-1})}.$$
This leads to the bound,
\begin{equation}\label{interpt}|T_{\Omega,\al}f(x)|\leq c_n\|\Omega\|_{L^{n,\infty}(\Sn^{n-1})}\sum_{k\in \Z}2^{\al k}\fint_{B(x,2^k)}|\nabla f|.\end{equation}
From here we switch our quantities from balls to cubes, namely, let $Q_k$ be the cube centered at $x$, with sidelength $2^{k+1}$ so that $B(x,2^k)\subseteq Q_k$.  Fix such a $Q_k$. Since $\ell(Q_k)=2^{k+1}$, using Lemma \ref{dyadicthird} we have that 
$$Q_k\subseteq Q \in \D_t^{k+4}$$
for some $t\in \{0,\frac13\}^n$.  Since $\D_t^{k+4}$ forms a partition of $\R^n$, we may use this estimate to get
\begin{align*}2^{\al k}\fint_{B(x,2^k)}|\nabla f| &\leq c_{\al,n}\ell(Q_k)^\al \fint_{Q_k}|\nabla f| \\ 
& \leq c_{n,\al}\sum_{Q\in \D_t^{k+4}} \ell(Q)^\al \left(\fint_{Q}|\nabla f|\right)\mathbf 1_Q(x)\\
&\leq c_{n,\al}\sum_{t\in \{0,\frac13\}}\sum_{Q\in \D_t^{k+4}} \ell(Q)^\al \left(\fint_{Q}|\nabla f|\right)\mathbf 1_Q(x).
\end{align*}
Using this estimate in inequality \eqref{interpt}, we have,
\begin{align*}
|T_{\Omega,\al}f(x)|&\leq c_n\|\Omega\|_{L^{n,\infty}(\Sn^{n-1})}\sum_{k\in \Z}2^{\al k}\fint_{B(x,2^k)}|\nabla f|\\
&\leq c_{n,\al}\|\Omega\|_{L^{n,\infty}(\Sn^{n-1})}\sum_{t\in \{0,\frac13\}^n}\sum_{k\in \Z}\sum_{Q\in \D_t^{k+4}} \ell(Q)^\al \left(\fint_{Q}|\nabla f|\right)\mathbf 1_Q(x)\\
&= c_{n,\al}\|\Omega\|_{L^{n,\infty}(\Sn^{n-1})}\sum_{t\in \{0,\frac13\}^n}\sum_{Q\in \D_t} \ell(Q)^\al \left(\fint_{Q}|\nabla f|\right)\mathbf 1_Q(x)\\
&\leq c_{n,\al}\|\Omega\|_{L^{n,\infty}(\Sn^{n-1})} I_\al(|\nabla f|)(x)
\end{align*}
where we have used the pointwise bound from Proposition \ref{dyadicfracbdds} in our last estimate.

{\bfseries Case 2:}
{Let $\frac1s=\frac1n+\frac1{r'}$ then $s^*=r'$ and $r^*=s'$. By the H\"older's inequality in Lorentz spaces, we have
\begin{align*}
\fint_{|y|\leq 2^k} & {|\Omega(y')|} |f(x-y)-f_{B(x,2^k)}|\,dy \\
& \leq \big\|\Omega(\cdot/|\cdot|)\big\|_{L^{r,r^*}(B(0,2^k))} \|f(x-\cdot)-f_{B(x,2^k)}\|_{L^{s^*,s}(B(0,2^k))} \\
& = c\|\Omega\|_{L^{r,r^*}(\mathbb{S}^{n-1})}\,\|f-f_{B(x,2^k)}\|_{L^{s^*,s}(B(x,2^k))}
\end{align*}
where we have again used the level set equality from \eqref{levelset} and the translation invariant property of the Lebesgue measure.  We now use the Poincar\'e-Sobolev inequality \eqref{Lorpoincare} for the oscillation term
$$\|f-f_{B(x,2^k)}\|_{L^{s^*,s}(B(x,2^k))}\leq C\, 2^k\left(\fint_{B(x,2^k)}|\nabla f|^s\right)^\frac1s.$$



Thus we have
\begin{align*}|T_{\Omega,\al}f(x)|&\leq C\,\|\Omega\|_{L^{r,r^*}(\Sn^{n-1})}\sum_{k\in \Z} 2^{\al k}\left(\fint_{B(x,2^k)}|\nabla f|^s\right)^{\frac1s}\\
&\leq C\,\|\Omega\|_{L^{r,r^*}(\Sn^{n-1})}\sum_{t\in \{0,\frac13\}^n}\sum_{Q\in \D_t} \ell(Q)^\al \left(\fint_Q |\nabla f|^s\right)^{\frac1s}\mathbf 1_Q(x)\\
&\leq  C\,\|\Omega\|_{L^{r,r^*}(\Sn^{n-1})}\sum_{t\in \{0,\frac13\}^n} I_{\al,L^s}^{\mathscr S_t}(|\nabla f|)(x)
\end{align*}
where the sparse families $\mathscr S_t\subseteq \D_t$ come from Lemma \ref{Lsdyadicfracbdds}.
}

{\bfseries Case 3:} In this case we use the Orlicz H\"older inequality \eqref{logexpholder} to estimate
\begin{align*}\lefteqn{\fint_{B(0,2^k)}|\Omega(y')||f(x-y)-f_{B(x,2^k)}|\,dy} \\
&\qquad \leq c\|\Omega(\cdot/|\cdot|)\|_{L(\log L)^{1/n'}(B(0,2^k))}\|f(x-\cdot)-f_{B(x,2^k)}\|_{\exp L^{n'}(B(0,2^k))}. 
\end{align*}
We now compute the $\|\Omega(\cdot/|\cdot|)\|_{L(\log L)^{\frac1{n'}}(B(0,2^k))}$ norm. Let $$\Psi(t)=t\log(1+t)^{\frac1{n'}}$$ and suppose $\lambda>0$ is such that 
$$\fint_{B(0,2^k)} \Psi\Big(\frac{|\Omega(y')|}{\lambda}\Big)\,dy\leq 1.$$
Let $v_n$ be the measure of the unit ball, then we have
\begin{align*}\fint_{B(0,2^k)} \Psi\Big(\frac{|\Omega(y')|}{\lambda}\Big)\,dy&=\frac1{v_n2^{kn}}\int_0^{2^k}\int_{\Sn^{n-1}}\Psi\Big(\frac{|\Omega(y')|}{\lambda}\Big)r^{n-1}\,drdy'\\
&=\frac{1}{nv_n}\int_{\Sn^{n-1}}\Psi\Big(\frac{|\Omega(y')|}{\lambda}\Big)\,dy'.
\end{align*}
It follows that
$$\|\Omega(\cdot/|\cdot|)\|_{L(\log L)^{\frac{1}{n'}}(B(0,2^k))}=c_n\|\Omega\|_{L(\log L)^{\frac1{n'}}(\Sn^{n-1})}.$$
Continuing the estimates and using the Trudinger inequality \eqref{trud}, we have
\begin{align*}
\lefteqn{\fint_{B(0,2^k)}|\Omega(y')||f(x-y)-f_{B(x,2^k)}|\,dy} \\ 
&\qquad \leq c\|\Omega\|_{L(\log L)^{\frac1{n'}}(\Sn^{n-1})}\|f-f_{B(x,2^k)}\|_{\exp L^{n'}(B(x,2^k))}\\
&\qquad \leq c\|\Omega\|_{L(\log L)^{\frac1{n'}}(\Sn^{n-1})}\left(\int_{B(x,2^k)}|\nabla f|^n\right)^{\frac1n}.
  \end{align*}
 By similar arguments to the first case, we may replace the sum with the sparse operators to obtain
 \begin{align*}
 |T_{\Omega,\al}f(x)|&\leq c\|\Omega\|_{L(\log L)^{\frac1{n'}}(\Sn^{n-1})} \sum_{k\in \Z} 2^{(\al-1)k}\left(\int_{B(x,2^k)}|\nabla f|^n\right)^{\frac1n}\\
 &=c\|\Omega\|_{L(\log L)^{\frac1{n'}}(\Sn^{n-1})} \sum_{k\in \Z} 2^{\al k}\left(\fint_{B(x,2^k)}|\nabla f|^n\right)^{\frac1n}\\
 &\leq c\|\Omega\|_{L(\log L)^{\frac1{n'}}(\Sn^{n-1})} \sum_{t\in \{0,\frac13\}}\sum_{Q\in \D_t}\ell(Q)^\al\left(\fint_{Q}|\nabla f|^n\right)^{\frac1n}\mathbf 1_Q(x)\\
 &=c\|\Omega\|_{L(\log L)^{\frac1{n'}}(\Sn^{n-1})} \sum_{t\in \{0,\frac13\}} I^{\D_t}_{\al,n}(|\nabla f|)(x)\\
 &\leq c\|\Omega\|_{L(\log L)^{\frac1{n'}}(\Sn^{n-1})} \sum_{t\in \{0,\frac13\}} I^{\mathscr S_t}_{\al,n}(|\nabla f|)(x),
 \end{align*}
where again we have used Lemma \ref{Lsdyadicfracbdds} to obtain the sparse bounds.

\begin{proof}[Proof of Theorem \ref{characterization}] Suppose that $T$ is an operator that satisfies
$$|Tf(x)|\leq cI_\al(|\nabla f|)(x),$$ for $x\in \R^n$ and $f\in C^\infty_c(\R^n).$
Let $\mu$ be a locally finite Borel measure and $f\in C_c^\infty(\R^n)$.  By the pointwise inequality and Minkowski's integral inequality we have
\begin{multline*}\|Tf\|_{L^{\frac{n}{n-\al},\infty}(\mu)}\leq C\|I_\al(|\nabla f|)\|_{L^{\frac{n}{n-\al},\infty}(\mu)} \\ \leq C\int_{\R^n}|\nabla f(y)|\big\||\cdot-y|^{\al-n}\big\|_{L^{\frac{n}{n-\al},\infty}(\mu)}\,dy.\end{multline*}
A calculation shows that
\begin{multline*}\big\||\cdot-y|^{\al-n}\big\|_{L^{\frac{n}{n-\al},\infty}(\mu)}=\sup_{\lambda>0}\lambda \mu(\{x:|x-y|^{\al-n}>\lambda\})^{\frac{n-\al}{n}}\\=\sup_{t>0}\Big(\frac{\mu(B(y,t))}{t^n}\Big)^{1-\frac{\al}{n}}\leq C(M\mu)(y)^{1-\frac{\al}{n}}.\end{multline*}

Now suppose the weak Sobolev inequality
$$\|Tf\|_{L^{\frac{n}{n-\al},\infty}(\mu)}\leq C\int_{\R^n}|\nabla f|(M\mu)^{1-\frac{\al}{n}}$$
holds for all locally finite Borel measures and $f\in C_c^\infty(\R^n)$.  Fix $x\in \R^n$ and let $\mu$ be the pointmass measure at $x$.  Then
$$(M\mu)(y)\simeq \frac{1}{|x-y|^n}.$$
Hence
\begin{multline*}|Tf(x)|=\|Tf\|_{L^{\frac{n}{n-\al},\infty}(\mu)}\leq C\int_{\R^n}|\nabla f(y)| (M\mu)(y)^{1-\frac{\al}{n}}\,dy\\ \leq C\int_{\R^n}\frac{|\nabla f(y)|}{|x-y|^{n-\al}}\,dy.\end{multline*}
\end{proof}

\section{Pointwise bounds for maximal operators} \label{maxbounds}
We now turn our attention to some results concerning the related maximal functions. As motivation, let us begin by addressing a natural question that arose in \cite{HMP}. Recall the following pointwise bound for the Hardy--Littlewood maximal operator:
\begin{equation}\label{HLbd}
Mf(x)\leq cI_1(|\nabla f|)(x), \qquad f\in C_c^\infty(\R^n).
\end{equation}

The Fefferman--Stein sharp maximal operator,
\[
M^\#f(x)=\sup_{Q\ni x}\fint_Q|f-f_Q|,
\]
satisfies a better bound:
\begin{equation}\label{FSsharpbd}
M^\#f(x)\leq cM_1(|\nabla f|)(x), \qquad f\in C_c^\infty(\R^n),
\end{equation}
where \( M_1 \) is the fractional maximal operator associated with the Riesz potential \( I_1 \). More generally, for \( 0<\alpha<n \), the fractional maximal operator is defined by
\[
M_\alpha f(x) = \sup_{Q\ni x} \ell(Q)^\alpha \left(\fint_Q |f|\right).
\]

Inequality \eqref{FSsharpbd} follows from the Poincar\'e--Sobolev inequality \eqref{poincare} in the case \( p = q = 1 \). Notably, the right-hand side of \eqref{FSsharpbd} improves upon that of \eqref{HLbd}, owing to the well-known pointwise inequality
\begin{equation} \label{MalIal}
M_\alpha f(x)\leq cI_\alpha f(x), \qquad f\geq 0.
\end{equation}

In light of \eqref{MalIal}, we see that the fractional maximal function \( M_\alpha \) satisfies the same \( L^p(\R^n) \) mapping properties as the Riesz potential operator, namely those in \eqref{Ialbds} and \eqref{weakIalbds}. However, $M_\alpha$ is better behaved at the endpoint $p = \frac{n}{\alpha}$:
\[
\|M_\alpha f\|_{L^\infty(\R^n)} \leq \|f\|_{L^{\frac{n}{\alpha}}(\R^n)}.
\]
(In fact, a sharper mapping holds: \( M_\alpha: L^{\frac{n}{\alpha},\infty}(\R^n) \to L^\infty(\R^n) \), but this is not needed for the present discussion.) 
In particular we have that $M^\#:\dot W^{1,n}(\R^n)\ra L^\infty(\R^n)$, since
\begin{equation}\label{BMOchar} \|M^\# f\|_{L^\infty(\R^n)}\leq c\|M_1(|\nabla f|)\|_{L^\infty(\R^n)}\leq C\|\nabla f\|_{L^n(\R^n)}.\end{equation}
Since $M^\#$ characterizes the space $BMO$, namely,
$$\|f\|_{BMO}=\|M^\#f\|_{L^\infty(\R^n)}$$
we see that inequality \eqref{BMOchar} is the well-known embedding 
$$\dot W^{1,n}(\R^n)\hookrightarrow BMO.$$
The space $W^{1,n}(\R^n)$ actually embeds into the smaller space $VMO$ (see \cite{BN}).

One might inquire about the potential for obtaining a better estimate for the Hardy-Littlewood maximal operator. Namely, is the following inequality 
\begin{equation}\label{falseineq} Mf(x)\leq cM_1(|\nabla f|)(x), \quad f\in C_c^\infty(\R^n)\end{equation}
true?  Inequality \eqref{falseineq} is, in fact, false. Indeed, if \eqref{falseineq} held, then we would have
\begin{equation}\label{falseendpt}\|f\|_{L^\infty(\R^n)}=\|Mf\|_{L^\infty(\R^n)}\leq C\|M_1(|\nabla f|)\|_{L^\infty(\R^n)}\leq C\|\nabla f\|_{L^n(\R^n)},\end{equation}
Inequality \eqref{falseendpt} is well-known to be false, for example, consider $f$ to be a smooth truncation of the function $f(x)=\log\big(\frac{e}{|x|}\big)^{1-\frac{2}{n-1}}$ for $|x|\leq 1$.  Put another way, the reason that \eqref{falseineq} cannot hold is because 
$$W^{1,n}(\R^n) \not\hookrightarrow L^\infty(\R^n).$$

Considering the homogeneity,
$$M_{\al-1}(f_\lambda)=\lambda^{1-\al}\big(M_{\al-1}f\big)_\lambda$$
the operator $M_{\al-1}$ defined for $1\leq \al<n$ is the natural maximal function associated to $I_\al(|\nabla f|)$.  The operator $M_{\al-1}$ also plays a role in the study of the Sobolev mappings of maximal operators.  Indeed, Kinnunen and Saksman \cite{KS}
show that when $\al\geq 1$, $M_\al$ has smoothing properties in the sense that
$$|\nabla M_\al f|\leq cM_{\al-1}f.$$
Observe that for $1\leq \al<n,$
\begin{equation} \label{Mal-1ptwise} M_{\al-1}f(x)\leq cI_\al(|\nabla f|)(x), \qquad  f\in C_c^\infty(\R^n).\end{equation}
 Indeed, when $\al=1$ this is just inequality \eqref{Mkptwise} with $k=1$.  For $1<\al<n$, we may use inequality \eqref{ptwisebdd} and the fact that $M_{\al-1}$ is also controlled by $I_{\al-1}$ \eqref{MalIal} to see that
\begin{equation*}M_{\al-1}f(x)\leq cI_{\al-1}(|f|)(x)\leq cI_\al(|\nabla f|)(x).\end{equation*}
In the remainder of this section, we will consider more exotic maximal operators with the same homogeneity as $M_{\al-1}$ and their bounds by potential operators of the gradient.

\subsection{Rough Maximal Operators}
The reason that $M^\#f$ is bounded by $M_1(|\nabla f|)$ but $M$ is not, stems from the presence of {cancellation} in $M^\#$. Namely $M^\#(\mathbf 1)=0$, whereas $M$ lacks this property. With this in mind, we now consider the companion rough maximal operators to our integral operators $T_{\Omega,\al}$.  Consider the rough maximal operator
$$M_{\Omega,\al}f(x)=\sup_{t>0} t^{\al-1}\fint_{|y|<t}|\Omega(y') f(x-y)|\,dy, \qquad 1\leq \al<n.$$
The operator $M_{\Omega,\al}$ has the same homogeneity as $T_{\Omega,\al}$ and $I_\al(|\nabla \cdot|)$.  When $\al=1$ this is the rough maximal operator associated with rough singular integrals (see \cite{Cr} and \cite{Duo}).  The operator $M_{\Omega,\al}$ does not have cancellation, and because of this, $M_{\Omega,\al}$ may not be well-defined in the hypersingular range $0<\al<1$.  Indeed, if $\Omega=\mathbf 1$, then $M_{\Omega,\al}(\mathbf 1_{B(0,1)})\equiv \infty$, near the origin if $\al<1$.  

We can consider a smaller maximal operator to introduce necessary cancellation, namely, consider the {\it natural} rough fractional maximal operator
$$M^\natural_{\Omega,\al}f(x)=\sup_{t>0} t^{\al-1}\left|\fint_{|y|<t}\Omega(y') f(x-y)\,dy\right|, \qquad 0<\al<n.$$
Clearly $M^\natural_{\Omega,\al}$ is smaller than $M_{\Omega,\al}$.  If $\int_{\Sn^{n-1}}\Omega=0$, then $M^\natural_{\Omega,\al}$ has cancellation in the sense that $M^\natural_{\Omega,\al}(\mathbf 1)=0$.  Finally we introduce a third maximal operator, the sharp rough fractional maximal operator
$$M^\#_{\Omega,\al}f(x)=\sup_{t>0} t^{\al-1}\fint_{|y|<t}|\Omega(y')||f(x-y)-f_{B(x,t)}|\,dy,\quad 0<\al<n.$$
In the definition of $M^\natural_{\Omega,\al}$ we will assume that $\Omega$ has mean zero, but this assumption is not needed for the other two maximal operators. We also notice that
\begin{equation} \label{ptwiseMnatsharp} M^\natural_{\Omega,\al}f(x)\leq M_{\Omega,\al}^\#f(x) \end{equation}
where we have used the zero average of $\Omega$ to subtract the constant $f_{B(x,t)}$.  In addition, the following pointwise bounds between $M_{\Omega,\al}^\#$ and $M_{\Omega,\al}$ hold.
\begin{theorem} Suppose $1\leq \al<n$ and $\Omega\in L^1(\Sn^{n-1})$, then the following pointwise bound holds
\begin{equation}\label{MOmegasharp}|M_{\Omega,\al}f(x)-M^\#_{\Omega,\al}f(x)|\leq \frac{1}{\omega_{n-1}}\|\Omega\|_{L^1(\Sn^{n-1})}M_{\al-1}f(x)\end{equation}
for all $x\in \R^n$, where $\omega_{n-1}=\mathcal H^{n-1}(\Sn^{n-1})$ is the surface area of the unit sphere.  In particular, if $f\in C_c^\infty(\R^n)$, 
\begin{equation*}|M_{\Omega,\al}f(x)-M^\#_{\Omega,\al}f(x)|\leq \frac{1}{\omega_{n-1}}\|\Omega\|_{L^1(\Sn^{n-1})}I_{\al}(|\nabla f|)(x).\end{equation*}
\end{theorem}
\begin{proof} Fix $x\in \R^n$ and $t>0$.  Notice that
\begin{align*}\lefteqn{ t^{\al-1}\fint_{|y|<t}|\Omega(y')||f(x-y)-f_{B(x,t)}|\,dy} \\
&\hspace{7mm} \leq  t^{\al-1}\fint_{|y|<t}|\Omega(y')||f(x-y)|dy+t^{\al-1}|f_{B(x,t)}|\fint_{|y|<t}|\Omega(y')|\,dy.
\end{align*}
If $v_n$ is the volume of the unit ball in $\R^n$ and $\omega_{n-1}=nv_n$ is the surface area of $\Sn^{n-1}$, then
$$\fint_{|y|<t}|\Omega(y')|\,dy=\frac{1}{v_nt^n}\int_0^t\int_{\Sn^{n-1}}|\Omega(y')|r^{n-1}d\sigma(y')dr=\frac1{\omega_{n-1}}\|\Omega\|_{L^1(\Sn^{n-1})}.$$
By taking the supremum over $t>0$, we have
$$M_{\Omega,\al}^\#f(x)\leq M_{\Omega,\al}f(x)+\frac1{\omega_{n-1}}\|\Omega\|_{L^1(\Sn^{n-1})}M_{\al-1}f(x).$$  On the other hand, again for a fixed $t>0$
\begin{multline*}{t^{\al-1}\fint_{|y|<t}|\Omega(y')f(x-y)|\,dy}\\ \qquad \leq{t^{\al-1}}\fint_{|y|<t}|\Omega(y')(f(x-y)-f_{B(x,t)})|dy\\+t^{\al-1}|f_{B(x,t)}|\fint_{|y|<t}|\Omega(y')|\,dy
\end{multline*}
leads to 
\begin{equation*} M_{\Omega,\al}f(x)\leq M_{\Omega,\al}^\#f(x)+\frac1{\omega_{n-1}}\|\Omega\|_{L^1(\Sn^{n-1})}M_{\al-1}f(x).\end{equation*}
which completes inequality \eqref{MOmegasharp}.  If $f$ is a smooth function, the bound by $I_\al(|\nabla f|)$ follows from inequality \eqref{Mal-1ptwise}.
\end{proof}

By these calculations and inequalities \eqref{ptwiseMnatsharp} and \eqref{MOmegasharp} we only need to investigate bounds for $M_{\Omega,\al}^\#$.  Our results for $M_{\Omega,\al}^\#$ now mirror those for $T_{\Omega,\al}$, except we can have a smaller maximal function on the right-hand side. We need an $L^s$ fractional maximal function,
\begin{equation}\label{Lsfracmax} M_{\al,L^s}f(x)=\sup_{Q\ni x}\ell(Q)^\al\Big(\fint_Q |f|^s\Big)^{\frac1s}, \qquad 1\leq s<\frac{n}{\al}.\end{equation}
Notice that $M_{\al,L^1}=M_\al$ and
$$M_{\al,L^s}f=\big[M_{\al s}(|f|^s)\big]^{\frac1s}.$$ 
We have the following theorem.

\begin{theorem} \label{maxfuncbdds} Suppose $0<\al<n$ and $\Omega\in L^1(\Sn^{n-1})$.  Then for $f\in C_c^\infty(\R^n)$ we have the following bounds
$$M_{\Omega,\al}^\#f(x)\leq \left\{\begin{array}{l} c\|\Omega\|_{L^{n,\infty}(\Sn^{n-1})}M_\al(|\nabla f|)(x),  \vspace{2mm} \\ c\|\Omega\|_{{L^{r,r^*}}(\Sn^{n-1})}M_{\al,L^{\frac{r'n}{r'+n}}}(|\nabla f|)(x), \vspace{2mm}\\ c\|\Omega\|_{L(\log L)^{\frac{1}{n'}}(\Sn^{n-1})}M_{\al,L^n}(|\nabla f|)(x),  \end{array}\right.$$
where $0<\al<1+\frac{n}{r'}$ in the second inequality and $0<\al<1$ in the third. 
\end{theorem}
We will not prove Theorem \ref{maxfuncbdds}, rather, we will only give a sketch since its proof uses similar calculations to those found in Section \ref{pointwisebds}.  All that is needed is the calculation at a single scale. Indeed, for a fixed $t>0$ by H\"older's inequality in the appropriate space $\mathcal X$
\begin{multline*}{t^{\al-1}}\fint_{|y|<t}|\Omega(y')(f(x-y)-f_{B(x,t)})|dy\\ \leq c\,t^{\al-1}\|\Omega \|_{\mathcal X(B(0,t))}\|f-f_{B(x,t)}\|_{\mathcal X'(B(x,t))}\end{multline*}
where $\mathcal X=L^{n,\infty}$, {$L^{r,r^*}$} for $1<r<n$, or $L(\log L)^{\frac{1}{n'}}$. The associate spaces now become $\mathcal X'=L^{n',1}$, {$L^{r',(r^*)'}$} for $1<r<n$, or $\exp(L^{n'})$ and we may use the Poincar\'e-Sobolev inequalities \eqref{Lorpoincare} or Trudinger's inequality \eqref{trud} to obtain the desired pointwise bounds.




While $M_\al$ and $I_\al$ have the same $L^p(\R^n)\ra L^q(\R^n)$ mapping properties when $1<p<\frac{n}{\al}$ and same weak endpoint behavior when $p=1$, they behave differently at the other endpoint $p=\frac{n}{\al}$. Likewise the same phenomenon holds for $M_{\al,L^s}$ and $I_{\al,L^s}$ when $s<\frac{n}{\al}$.  As mentioned above $M_\al:L^{\frac{n}{\al}}(\R^n)\ra L^\infty(\R^n)$, and it is a simple consequence of H\"older's inequality that
$$M_{\al,L^s}:L^{\frac{n}{\al}}(\R^n)\ra L^\infty(\R^n).$$ 
An interesting consequence is the following $L^\infty$ Sobolev bounds, which we state when $\al=1$ for simplicity.

\begin{corollary} Suppose $\Omega \in L^{r,\infty}(\Sn^{n-1})$ for some $1<r<\infty$.  If $f\in C_c^\infty(\R^n)$ then 
$$\|M^\natural_\Omega f\|_{L^\infty(\R^n)}\leq \|M^\#_\Omega f\|_{L^\infty(\R^n)}\leq C\|\nabla f\|_{L^n(\R^n)}.$$
\end{corollary}


\subsection{The Fractional Spherical Maximal Operator}

In this section, we extend the pointwise bounds to fractional spherical maximal functions.  Kinnunen and Saksman \cite{KS} defined the fractional spherical maximal function for $0\leq \beta<n-1$ by
$$\Sp_\beta f(x)=\sup_{r>0} r^\beta \fint_{\partial B(x,r)}|f(y)|d\mathcal H^{n-1}(y).$$ 
In our context, we will work with $\beta=\al-1$, where $1\leq \al<n$. This way the operators $\Sp_{\al-1}$ match the homogeneity of the operators $T_{\Omega,\al}$ and $M_{\Omega,\al}$. We also point out that $\Sp_{\al-1}$ arises in the study of the smoothness of the fractional maximal function.  Indeed, it was shown in \cite{KS} that
$$|\nabla M_\al f(x)| \leq c_n \Sp_{\al-1}f(x).$$
We have the following pointwise bounds.
\begin{theorem} Suppose $1\leq \al<n$, then the following pointwise bound holds
$$\Sp_{\al-1}f(x)\leq c_n I_\al(|\nabla f|)(x).$$
\end{theorem}
\begin{proof} Let $f\in C_c^\infty(\R^n)$ and given $\ep>0$ define the positive and bounded $C^\infty(\R^n)$ function
$$g_\ep(x)=(f(x)^2+\ep)^{\frac12}.$$
The gradient of $g_\ep$ is given by 
$$\nabla g_\ep(x)=\frac{f(x)\nabla f(x)}{(f(x)^2+\ep)^{\frac12}}$$
and satisfies
$$|\nabla g_\ep(x)|\leq |\nabla f(x)|.$$
Fix $x\in \R^n$, $t>0$ and let $\sigma =\mathcal H^{n-1}$ be the surface measure on $\partial B(0,1)=\Sn^{n-1}$ where again $\omega_{n-1}=\sigma(\Sn^{n-1})$.  Then we have
\begin{align*}
\fint_{\partial B(x,t)}g_\ep(y)d\mathcal H^{n-1}(y)&= \frac{1}{\omega_{n-1}t^{n-1}}\int_{\partial B(0,t)}g_\ep(x-y)\,d\mathcal H^{n-1}(y)\\
&=\frac{1}{\omega_{n-1}}\int_{\Sn^{n-1}}g_\ep(x-t\theta)\,d\sigma(\theta).
\end{align*} 
Let $N>t$ be sufficiently large so that $f(x-N\theta)=0$ for all $\theta\in \mathbb S^{n-1}$.  This implies $g_\ep(x-N\theta)=\sqrt{\ep}$ for all $\theta\in \mathbb S^{n-1}$. A similar calculation shows
$$\fint_{\partial B(x,N)}g_\ep(y)d\mathcal H^{n-1}(y)=\frac{1}{\omega_{n-1}}\int_{\Sn^{n-1}}g_\ep(x-N\theta)\,d\sigma(\theta)=\sqrt{\ep}$$
and hence
$$\lim_{N\ra \infty}\fint_{\partial B(x,N)}g_\ep(y)d\mathcal H^{n-1}(y)=\sqrt{\ep}.$$
We now have
\begin{align*}
\fint_{\partial B(x,t)}g_\ep(y)d\mathcal H^{n-1}(y)&=-\frac{1}{\omega_{n-1}}\int_t^\infty \frac{d}{dr}\int_{\Sn^{n-1}}g_\ep(x-r\theta)\,d\sigma(\theta)dr+\sqrt{\ep} \\
&=\frac{1}{\omega_{n-1}}\int_t^\infty \int_{\Sn^{n-1}}\nabla g_\ep(x-r\theta)\cdot \theta\,d\sigma(\theta)dr+\sqrt{\ep}\\
&\leq \frac{1}{\omega_{n-1}}\int_t^\infty\int_{\Sn^{n-1}} |\nabla g_\ep(x-r\theta)| d\sigma(\theta)dr+\sqrt{\ep}\\
&=\frac{1}{\omega_{n-1}}\int_{|y|>t} \frac{|\nabla g_\ep(x-y)|}{|y|^{n-1}}\,dy+\sqrt{\ep}\\
&\leq \frac{1}{\omega_{n-1}}\int_{|y|>t} \frac{|\nabla f(x-y)|}{|y|^{n-1}}\,dy+\sqrt{\ep}.
\end{align*} 
Letting $\ep\ra 0^+$ we obtain
$$\fint_{\partial B(x,t)}|f(y)|d\mathcal H^{n-1}(y)\leq c_n\int_{|y|>t} \frac{|\nabla f(x-y)|}{|y|^{n-1}}\,dy$$
and hence
\begin{multline*}t^{\al-1} \fint_{\partial B(x,t)}|f(y)|d\mathcal H^{n-1}(y)\leq c_nt^{\al-1}\int_{|y|>t} \frac{|\nabla f(x-y)|}{|y|^{n-1}}\,dy\\
\leq c_n\int_{|y|>t} \frac{|\nabla f(x-y)||y|^{\al-1}}{|y|^{n-1}}\,dy\leq c_nI_\al(|\nabla f|)(x).\end{multline*}
\end{proof}
As a consequence, we have the following Sobolev bounds for $\Sp_{\al-1}$, which seem to be new. Corollary \ref{sobolevsphere} will follow from the mapping properties of $I_\al$.

\begin{corollary} \label{sobolevsphere} Suppose $1\leq \al<n$ and $1<p<\frac{n}{\al}$ and $\frac1q=\frac1p-\frac{\al}{n}$.  Then $\Sp_{\al-1}:\dot W^{1,p}(\R^n)\ra L^q(\R^n)$ with
$$\|\Sp_{\al-1}f\|_{L^q(\R^n)}\leq C\|\nabla f\|_{L^p(\R^n)}.$$
At the endpoint $p=1$ we have the following weak-type estimate
$$\|\Sp_{\al-1}f\|_{L^{\frac{n}{n-\al},\infty}(\R^n)}\leq C\|\nabla f\|_{L^1(\R^n)}.$$
\end{corollary}

\section{Weighted estimates} \label{weights}
In this section, we prove new weighted estimates for our sparse operators
$$I_{\al,L^s}^\mathscr S f=\sum_{Q\in \mathscr S}\ell(Q)^\al\left(\fint_Q f^s\right)^{\frac1s}\mathbf 1_Q, \qquad f\geq 0, $$
where $0<\al<n$ and $1\leq s<\frac{n}{\al}.$
Since these operators dominate our rough integral operators $T_{\Omega,\al}$ (Theorems \ref{poinwisecriticalT}, \ref{rlessnptwisebd}, and \ref{LlogLendpt}) we will also obtain new weighted Sobolev estimates of the form
\begin{equation}\label{oneweightSob}\|wT_{\Omega,\al}f\|_{L^q(\R^n)}\leq C\|w\nabla f\|_{L^p(\R^n)}.\end{equation}
We will refer to inequality \eqref{oneweightSob} as a weighted Sobolev inequality. The parameters and class of weights will depend on the size of $\Omega$.

As mentioned in the introduction, the operator $I_{\al,L^s}^\mathscr S$ is intimately related to the operator $f\mapsto \big[I_{\al s}(|f|^s)\big]^{\frac1s}$ (see Theorem \ref{Ainftybdds}).  We can work with the smaller maximal function $M_{\al,L^s}$ defined in \eqref{Lsfracmax}. Notice that 
$$M_{\al,L^s}f=\big[M_{\al s}(|f|^s)\big]^{\frac1s}\leq c\big[I_{\al s}(|f|^s)\big]^{\frac1s}.$$
By inequality \eqref{integralsparsebd} there exists finitely many sparse families $\mathscr S_1,\ldots,\mathscr S_N$ such that
$$M_{\al,L^s}f\leq c\sum_{k=1}^N I_{\al,L^s}^{\mathscr S_k}|f|.$$
The following theorem shows that we may obtain the opposite inequalities in norm.  It also yields Theorem \ref{Ainftybdds} as a corollary.

\begin{theorem} \label{AinftyMalLs} Suppose $1\leq s<\frac{n}{\al}$, $w\in A_\infty$, and $0<p<\infty$.  Then given a sparse family $\mathscr S$ of cubes there exists $C>0$, such that
$$\|I_{\al,L^s}^\mathscr S f\|_{L^p(w)}\leq C\|M_{\al,L^s} f\|_{L^p(w)}$$
and
$$\|I_{\al,L^s}^\mathscr S f\|_{L^{p,\infty}(w)}\leq C\|M_{\al,L^s} f\|_{L^{p,\infty}(w)}.$$  
\end{theorem}
\begin{proof} To prove this theorem, we will first prove it in the case $p=1$ for any $w\in A_\infty$. Once this is shown the conclusion will then follow from the $A_\infty$ version of the Rubio de Francia extrapolation (see \cite[Theorem 1.1]{CMP}).  Fix a sparse subset $\mathscr S$ and $w\in A_\infty$. 
Let $E_Q\subseteq Q\in \mathscr S$ be pairwise disjoint majorizing sets.  Since $w\in A_\infty$ and $|Q|\leq 2|E_Q|$ we have 
$$w(Q)\leq Cw(E_Q).$$
We can now estimate the $L^1(w)$ norm 
\begin{align*}
\int_{\R^n}|I_{\al,L^s}^\mathscr S f| w &\leq \sum_{Q\in \mathscr S} \ell(Q)^\al \left(\fint_Q |f|^s\right)^{\frac1s} w(Q)\\
&\leq C\sum_{Q\in \mathscr S} \ell(Q)^\al \left(\fint_Q |f|^s\right)^{\frac1s} w(E_Q)\\
&\leq C\sum_{Q\in \mathscr S} \int_{E_Q}\big(M_{\al,L^s}f\big)w\\
&\leq C\int_{\R^n} \big(M_{\al,L^s}f\big)w.
\end{align*}
This concludes the proof of Theorem \ref{AinftyMalLs}.
\end{proof}
Recall that the weighted estimates for $I_\al$ and $M_\al$ are known and were discovered by Muckenhoupt and Wheeden \cite{MW}.  They showed the following equivalences for $1<p<\frac{n}{\al}$ and $\frac1q=\frac1p-\frac\al{n}$:
\begin{itemize}
\item $M_\al$ is bounded:
$$\|wM_\al f\|_{L^q(\R^n)}\leq C\|wf\|_{L^p(\R^n)};$$
\item $I_\al$ is bounded:
$$\|wI_\al f\|_{L^q(\R^n)}\leq C\|wf\|_{L^p(\R^n)};$$
\item $w\in A_{p,q}$
$$\sup_Q \left(\fint_Q w^q\right)^\frac1q\left(\fint_Q w^{-p'}\right)^{\frac1{p'}}<\infty.$$
\end{itemize}
The weak-type boundedness at the endpoint $p=1$, $q=\frac{n}{n-\al}$
$$M_\al,I_\al :L^1(w)\ra L^{\frac{n}{n-\al},\infty}(w)$$
is equivalent to the $w\in A_{1,\frac{n}{n-\al}}$ condition:
$$ \left(\fint_Q w^\frac{n}{n-\al}\right)\leq C\inf_{x\in Q} w(x)^{\frac{n}{n-\al}}.$$
The $A_{p,q}$ classes of weights are related to the usual Muckenhoupt class of weights via the following realization:
\begin{equation}\label{Apq} w\in A_{p,q} \Leftrightarrow w^q\in A_{1+\frac{q}{p'}},\end{equation}
in particular $A_{p,q}$ weights raised to the power $q$ belong to $A_\infty$.  Using this rescaling we can easily obtain the weighted norm inequalities for $I_{\al,L^s}^\mathscr S$ and $M_{\al,L^s}$ by realizing that it is just a rescaling of the $L^p$ norm:
$$\|wM_{\al,L^s}f\|_{L^q(\R^n)}=\left(\int_{\R^n}\big[M_{\al s}(|f|^s)\big]^{\frac{q}{s}}w^q\right)^{\frac1q}=\|w^{s}M_{\al s}(|f|^s)\|_{L^{\frac{q}{s}}(\R^n)}^{\frac1s}.$$
Since 
$$\frac1q=\frac1p-\frac{\al}{n} \Leftrightarrow \frac{s}{q}=\frac{s}{p}-\frac{\al s}{n}$$
we obtain the following theorem.

\begin{theorem} \label{oneweightthm} Suppose $1\leq s<p<\frac{n}{\al}$ and $q$ is defined by $\frac1p-\frac1q=\frac{\al}{n}$.  The following are equivalent:
\begin{enumerate}
\item $w^s \in A_{\frac{p}{s},\frac{q}{s}}$:
$$\sup_Q\left(\,\fint_Q w^q\right)^\frac1q\left(\,\fint_Q w^{-\frac{ps}{p-s}}\right)^{\frac{p-s}{ps}}<\infty;$$
\item the operator $M_{\al,L^s}$ satisfies
$$\|wM_{\al,L^s} f\|_{L^q(\R^n)}\leq C\|wf\|_{L^p(\R^n)};$$
\item for any sparse family of cubes $\mathscr S$, we have
$$\|wI_{\al,L^s}^\mathscr S f\|_{L^q(\R^n)}\leq C\|wf\|_{L^p(\R^n)};$$
\item the operator $f\mapsto \big[I_{\al s}(|f|^s)\big]^{\frac1s}$ satisfies
$$\big\|w\big[I_{\al s}(|f|^s)\big]^{\frac1s}\big\|_{L^q(\R^n)}\leq C\|wf\|_{L^p(\R^n)}.$$
\end{enumerate}
When $p=s$ and $q=\frac{ns}{n-\al s}$ any of the operators $I_{\al,L^s}^\mathscr S$, $M_{\al,L^s}$, or $[I_{\al s}(|\cdot|^s)]^{\frac1s}$ are bounded from $L^s(w^s)\ra L^{q,\infty}(w^q)$ when $w^s\in A_{1,\frac{q}{s}}$ i.e.,
$$\left(\fint_Qw^{\frac{ns}{n-\al s}}\right)\leq C\inf_{x\in Q} w(x)^{\frac{ns}{n-\al s}}$$
for all cubes $Q\subseteq \R^n$.
\end{theorem}

When $\Omega \in L^{r,r^*}(\Sn^{n-1})$ or $\Omega \in L(\log L)^{\frac1{n'}}(\Sn^{n-1})$, the pointwise bounds from Theorems \ref{rlessnptwisebd} and \ref{LlogLendpt} combined with Theorem \ref{oneweightthm} will yield new weighted Sobolev estimates for $T_{\Omega,\al}$. The class of weights will depend on the exponent $s=\frac{r'n}{r'+n}$.  While this exponent can be used to compute the explicit class of weights, we find it more instructive to consider the power weights.  It is well-known that the power weight $w(x)=|x|^\lambda$ belongs to $A_p$ if and only if $-n<\lambda<n(p-1)$.  Using the equivalence \eqref{Apq} we see that $w(x)=|x|^\lambda$ belongs to $A_{p,q}$ if and only if 
$$-\frac{n}{q}<\lambda<\frac{n}{p'}.$$
Such power weights satisfy $w^s\in A_{\frac{p}{s},\frac{q}{s}}$ if and only if
$$-\frac{n}{q}<\lambda<\frac{n}{s}-\frac{n}{p}.$$
When $s=\frac{r'n}{r'+n}$, by our pointwise bounds from Theorem \ref{rlessnptwisebd} we obtain Theorems \ref{roughpower1} and \ref{roughpower2}.



\end{document}